\documentclass[a4paper,11pt]{article}
\usepackage{amsmath, amsthm, amscd, amsfonts, amssymb,color,graphicx}

\newtheorem*{theorem*}{Theorem}

\newtheorem*{corollary*}{Corollary}

\newtheorem*{definition*}{Definition}

\numberwithin{equation}{section}

\begin{document}
\title{\textbf{On non-maximal prime ideals of $C[0,1]$}}
\author{Vaibhav Pandey (vaibhav2011@niser.ac.in)\\ NISER, Bhubaneswar, India}
\date{}
\maketitle

\begin{abstract}
We first show a counter intuitive result that in the ring of real valued continuous functions on $[0,1]$ non maximal prime ideals exist. This is a standard proof and a well known result. Interestingly, a non maximal prime ideal in this ring is actually contained inside a unique maximal ideal. We arrive at this result merely by looking at the zero set of ideals in this ring and by making simple geometrical observations. We end by leaving the reader with an interesting open problem that logically follows from this article.
\end{abstract}

\section*{Introduction}

Consider the ring $R = C[0,1] = \{f:[0,1] \rightarrow \mathbb{R} : \textit{f is continuous}\}$ under pointwise addition and multiplication. Consider any two nonzero functions $f$ and $g$ in $R$ whose zero sets are complements of each other in $[0,1]$, then the product $fg$ is the zero function. So, $C[0,1]$ is not an integral domain.

Further, we know all the maximal ideals of $R$. All maximal ideals of $R$ are of the form $ M_{\gamma}$ for some $\gamma \in [0,1]$ where $M_{\gamma} = \{f \in R : f(\gamma) = 0\}$. Now, given the points $\gamma_1, \gamma_2 \in [0,1]$, consider the ideal $I = \{f \in R : f(\gamma_1) = f(\gamma_2) = 0 \}$. Is it a prime ideal? No. This is because the polynomial $(x-\gamma_1)(x-\gamma_2) \in I$ but both $x-\gamma_1$ and $x-\gamma_2$ do not individually belong to $I$.

\section*{Towards the question}

We ask if non maximal prime ideals exist in $R$ and if so where can we locate them. From the above discussion it appears that ideals vanishing at $2$ or more points are not prime ideals. This gives us a nice tool of associating with a given ideal $I$ of $R$, a subset of $[0,1]$ namely $V(I) = \bigcap_{f \in I} V(f)$ where $V(f) = f^{-1}(0)$. For example $V(M_{\gamma}) = \{\gamma\}$. Note that $V(I)$ is compact in $[0,1]$. Also note that if $I_1 \subset I_2$ then $V(I_2) \subset V(I_1)$. We first investigate the existence of non maximal prime ideals in $R$.    

\section*{An existential proof}

Since ideals $I$ with $|V(I)| = 2$ are not prime ideals, one would guess that ideals with $|V(I)| \geq 2$ are not prime ideals as well. Based on this, our intuition might tell us that $R$ has no non maximal prime ideals which would be rather interesting since $R$ is not a principal ideal domain (in fact not even an integral domain). However, we show the existence of non maximal prime ideals. But first a definition.

\begin{definition*}
A nonempty set $S$ is said to be multiplicative if $1 \in S$ and given any two elements in $S$, the product of these two elements also lies in $S$.
\end{definition*}

\textbf{Proof} - Let $S$ be the set of all polynomials in $C[0,1]$. Note that $S$ is a multiplicative set. Now consider all the ideals in $R$ with the property that they are disjoint from $S$. Call this set $A$ with the partial ordering of set inclusion. Note that $A$ is nonempty since the zero ideal belongs to it. Consider a chain $I_1 \subset I_2 \subset \ldots I_n \ldots$ of ideals in $A$. Then the ideal $\bigcup I_j$ is clearly an upper bound. By Zorn's Lemma, $A$ has a maximal element. Call it $P$.\\
We claim that $P$ is a prime ideal. If not, then there exist $a,b \in R$ (outside $P$) such that $ab \in P$. Consider the ideals $\langle P,a \rangle$ and $\langle P,b \rangle$. Both these ideals strictly contain $P$ and therefore must intersect with $S$. Hence, there exist $f,g \in R$ and $p,p' \in P$ such that $p+fa, p'+gb \in S$. As $S$ is multiplicative $(p+fa)(p'+gb) \in S$. Now $(p+fa)(p'+gb) = pp' + pgb + p'fa + fgab$. As $P$ is an ideal $pp', pgb, p'fa \in P$. By assumption $ab \in P$ so $fgab \in P$. Therefore, we get that $(p+fa)(p'+gb) \in P$. But this is a contradiction to the fact that $S \cap P = \phi$. Hence, $P$ is a prime ideal and the claim is proved.\\
We further claim that $P$ is not a maximal ideal. We prove this by contradiction. If $P$ is a maximal ideal then $P = M_{\gamma}$ for some $\gamma \in [0,1]$. Then the nonzero polynomial $p(x) = x - \gamma \in P$. But $P$ does not contain any nonzero polynomials. This contradiction proves our claim that $P$ is not a maximal ideal.

\section*{Locating non maximal prime ideals}

We now know that non maximal prime ideals do exist in $R$. We ask how does the zero set $V(P)$ of a non maximal prime ideal $P$ look like.

\begin{theorem*}
If $P$ is a prime ideal of $R$ then $V(P)$ cannot be the empty set or a set with more than $1$ element.
\end{theorem*}
\begin{proof}
Since $R$ is a ring with $1$, any ideal $I$ is contained in a maximal ideal. So, $I \subset M_{\gamma}$ for some $\gamma \in [0,1]$. Then, $|V(I)| \geq |V(M_{\gamma})| = 1$. So, $|V(I)| = 0$ is not possible for any ideal $I$ of $R$, and hence in particular not possible for prime ideals of $R$.\\
Now consider the case when $|V(P)| \geq 2$. Let $\gamma_1, \gamma_2 \in V(P)$ with $\gamma_1 < \gamma_2$. Find points $x_1$ and $x_2$ in $(\gamma_1, \gamma_2)$ such that $x_1 < x_2$. Define 
\[ f(x) = \left\{
\begin{array}{ll}
      0 & x\leq x_2 \\
      x-x_2 &  x\geq x_2 \\
\end{array} 
\right. \] and 
\[ g(x) = \left\{
\begin{array}{ll}
      x-x_1 & x\leq x_1 \\
      0 &  x\geq x_1 \\
\end{array} 
\right. \] 
Then $f(\gamma_2) \neq 0$ and $g(\gamma_1) \neq 0$, so $f$ and $g$ do not belong to $I$ whereas $fg = 0$ belongs to $I$. Thus $I$ is not a prime ideal.
\end{proof}
Note that the above argument works equally well regardless of whether $V(P)$ is a finite set or an infinite set.

\section*{Conclusion}

Since non maximal prime ideals exist in $R$, from the above theorem it directly follows that for such an ideal $P$, $|V(P)| =1$ as all other possibilities are eliminated.
\begin{corollary*}
A non maximal prime ideal of $C[0,1]$ is (strictly)contained in a unique maximal ideal ($M_\gamma$). 
\end{corollary*}
In conclusion, we remark that we just showed the existence of non maximal prime ideals and located them in some sense. It would be a good follow up if someone can come up with a constructive proof of a non maximal prime ideal in this ring. One can prove that the maximal ideals of $C[0,1]$ are actually unaccountably generated [$\textbf{3}$, p. 404] and it is probably difficult to come up with an explicit set of generators for these ideals. Therefore, it doesn't seem a very easy job to find a generating set for non maximal prime ideals in this ring.

\end{document}